\documentclass[10pt,a4paper]{article}
\usepackage{amsmath}                    
\usepackage{amssymb}                    
\usepackage{amsthm}                     
\usepackage{amsfonts}                   
\usepackage{subcaption}
\usepackage[usenames, dvipsnames]{color}
\usepackage{graphics,graphicx}
\usepackage{enumerate}
\usepackage{multirow}

\usepackage{dsfont}
\usepackage{float}
\usepackage{mathtools}
\usepackage[hidelinks]{hyperref}

\allowdisplaybreaks

\newtheorem{theorem}{Theorem}[section]

\newtheorem{lemma}[theorem]{Lemma}
\newtheorem{remark}[theorem]{Remark}

\newtheorem{algo}[theorem]{Algorithm}

\def\ts{\thinspace}

\renewcommand{\epsilon}{\varepsilon}

\newcommand{\R}{\mathbb{R}}
\newcommand{\bigo}{\mathcal{O}}

\def \IR{\mathbb R}
\newcommand{\EE}{\IR^d}
\def \IN{\mathbb N}
\def \IC{\mathbb C}

\newcommand{\calS}{\mathcal{S}}

\renewcommand{\R}{\mathbb{R}}

\usepackage[symbol]{footmisc}

\title{
	A fully adaptive explicit stabilized integrator for advection-diffusion-reaction problems
}

\author{ 
	Ibrahim Almuslimani\textsuperscript{*}
}

\begin{document}
	\maketitle
	\footnotetext[1]{Univ Rennes, INRIA Rennes, IRMAR - UMR 6625, F-35000 Rennes, France. \\ \hspace*{7mm}Ibrahim.Almuslimani@univ-rennes1.fr.}

\begin{abstract}
A novel second order family of explicit stabilized Runge--Kutta--Chebyshev methods for advection--diffusion--reaction equations is introduced. The new methods outperform existing schemes for relatively high Peclet number due to their favorable stability properties and explicitly available coefficients. The construction of the new schemes is based on stabilization using second kind Chebyshev polynomials first used in the construction of the stochastic integrator SK-ROCK. An adaptive algorithm to implement the new scheme is proposed. This algorithm is able to automatically select the suitable step size, number of stages, and damping parameter at each integration step. Numerical experiments that illustrate the efficiency of the new algorithm are presented.
		
\smallskip
\noindent
{\it Keywords:\,}
advection-diffusion-reaction equations, explicit stabilized methods, Runge-Kutta Chebyshev methods, RKC, SK-ROCK, ARKC.
\smallskip

\noindent
{\it AMS subject classification (2010):\,}
65L04, 65L20, 65M12
\end{abstract}

\section{Introduction}
In this paper we use the idea of stabilization by combining first and second kind Chebyshev polynomials introduced in \cite{AAV18} to derive explicit stabilized methods for advection--diffusion problems with, possibly, costly non-stiff reaction terms,
\begin{equation*}
	\partial_t u(x,t)= \nabla\cdot( D\nabla u(x,t))-\nabla\cdot(\textbf vu(x,t))+r(u(x,t)),\quad (x,t)\in \Omega\times [0,T],
\end{equation*}
with initial and boundary conditions, where $\Omega\in\EE$, $D$ is the matrix of diffusion coefficients, and $\textbf v$ is the velocity vector. The function $r$ represents non-stiff, but possibly costly, reaction terms. Note that in general, $D$ and $\textbf v$ may also depend on $u$ leading to nonlinear diffusion and advection terms. In the linear one dimensional setting, the equation reduces to
\begin{equation}\label{eq:ADlin}
	\partial_t u(x,t)= d\partial ^2_xu(x,t)-a\partial_xu(x,t)+r(u(x,t)),\quad (x,t)\in \Omega\times [0,T]
\end{equation}
where $d$ and $a$ are positive reals, and $\Omega$ is a real interval. The Peclet number is defined by $a/d$ and is allowed here to be quite large. When discretizing the partial differential equation (PDE) \eqref{eq:ADlin} in space using centered finite difference for example, with mesh size $\Delta x$, We obtain a system of ordinary differential equations (ODEs) of the form
\begin{equation}\label{eq:odead}
	\dot y(t)=F_D(y(t))+F_A(y(t)),\qquad y(0)=y_0\in\EE,\qquad t\in[0,T],
\end{equation}
where $F_D$ represents the diffusion term with eigenvalues of its Jacobian grow as $1/\Delta x^2$ on the negative real axis, and $F_A$ represents the advection term (and possibly non-stiff reaction terms) with eigenvalues of its Jacobian are of size $1/\Delta x$ and located close to the imaginary axis and symmetric with respect to the origin. This means that the eigenvalues of the Jacobian of the obtained system are approximately located in an ellipse with the length of its minor axis proportional to the square root of the length of the major axis.

Explicit stabilized Runge--Kutta--Chebyshev methods were originally introduced in the context of purely diffusive or diffusion dominated advection--diffusion problems (very small Peclet number) as a compromise between costly implicit methods and restrictive usual explicit schemes \cite{Abd01,Abd02,Abd13c,AbM01,HoS80}. Due to their versatility, they were extended to many other types of problems such as advection--diffusion--reaction equations {\cite{AbV12a,TX20,VS04,VSH04,Zb11}}, stochastic differential equations (SDEs) {\cite{AAV18,AbdS20,AbL08,AVZ13}}, and optimal control problems \cite{AV19}. {Other types of stabilized methods were studied in \cite{TJ07a,TJ07b,KA12}.}

Typically, the stability domain of an explicit stabilized method contains a long narrow strip around the negative real axis. In the context of advection--diffusion problems, the authors of the article \cite{VSH04} propose the usage of the RKC method with very large damping parameter to make the strip wider, which means that eigenvalues with slightly larger imaginary parts coming from the advection terms can be put in. This comes at the cost of a serious shortening of the strip, which means that, for a fixed time step, less eigenvalues with negative real parts can be put in. Later, a partitioned Runge--Kutta--Chebyshev method (PRKC) of order $2$ was designed in \cite{Zb11} based on the RKC method \cite{SSV98} for the integration of ODEs that have a moderately stiff term (diffusion) and non-stiff terms (advection or costly reaction terms). PRKC has a limited stability for the advection term, and it shares with the standard RKC the same stability domain length over the negative real axis. In \cite{AbV12a}, the authors propose a partitioned implicit--explicit orthogonal Runge--Kutta method (called PIROCK) for the time integration of advection--diffusion--reaction problems with possibly severely stiff reaction terms and stiff stochastic terms. The diffusion terms are solved by the explicit, nearly optimal, second order orthogonal Chebyshev method (ROCK2). Applied to advection--diffusion problems, the method has order 2 of accuracy and can handle the large Peclet number regime but it needs very large damping that reduces a lot its stability domain length over the negative real axis. In addition, PIROCK relies on the ROCK2 method, for which no explicit formulas are available to compute the coefficients for a given stage number.
{Recently, the authors of \cite{TX20} developed an improved version of the RKC method (called IMPRKC) for advection--diffusion--reaction equations. Their idea is based on introducing an appropriate combination of RKC polynomials which leads to a significant increase of the width of the stability domain along the imaginary axis with almost no loss of its length along the real axis. This comes at the cost of a few additional function evaluations. The main drawback of the IMPRKC scheme is evaluating $F_A$ at each stage of the method.}
This rich literature shows that the domain of stabilized schemes is very active, and that the construction of an adaptively efficient explicit stabilized integrator for such important class of problems is quite challenging.

In \cite[Sect.\ts3.7.2]{A20}, the author of the present paper profited from the idea of stabilization using second kind Chebyshev polynomials, first introduced in \cite{AAV18} {in the context of SDEs}, to construct a first order explicit stabilized method for advection--diffusion--reaction equations with optimal stability domain. What made this quite intuitive is the similarity between mean square stability for SDEs and the stability of the test equation for ODEs of the form \eqref{eq:odead} (see Sect. \ref{sec:stab}). In this work, we construct a second order integrator based on RKC that outperforms the existing  methods in the literature. We propose a fully adaptive algorithm to implement the new second order method.

This paper is organized as follows: in Section \ref{sec:esm}, we give a fast revision on explicit stabilized methods. In Section \ref{sec:o1} we present an \emph{optimal} first order explicit stabilized method for advection--diffusion--reaction equations that we constructed in the thesis \cite{A20} inspired by previous work on SDEs \cite{AAV18}. In Section \ref{sec:o2} we derive and analyze our new second order adaptive scheme in terms of stability and convergence, and we propose a local error estimator for automatic step size selection. Section \ref{sec:num} is dedicated to present and analyze some numerical experiments that illustrate the efficiency of the new schemes. {Finally, we conclude in Section \ref{sec:conclusion}.}

\section{Preliminaries on explicit stabilized methods}\label{sec:esm}
This section is devoted to present useful standard materials.

In order to study the stability of a Runge--Kutta integrator applied to an ODE, the following approach is widely used \cite{Abd13c,HaW96}. Consider the test ODE
\begin{equation}\label{eq:odetest}
	\dot y(t) = \lambda y(t),\quad y(0)=y_0, 
\end{equation}
where $\lambda\in\IC$ with negative real part. If we apply a Runge--Kutta method with step size $h$ on \eqref{eq:odetest}, we get the relation $y_{n+1}=R(h\lambda)^ny_0$, where the rational function $R(z)$ is called the stability function. Hence, the stability domain of the method is defined as
\begin{equation*}
	{\calS}:=\{z\in \mathbb{C}; |R(z)|\leq 1\}.
\end{equation*}
In the particular case of explicit methods, $R(z)$ is a polynomial which means that the stability domain is necessarily bounded. For example, the stability domain of the explicit Euler method is just a disk of radius $1$ which explains the time step restriction it faces for stiff ODEs, while that of the implicit Euler method is the complementary of a disk of radius $1$. This shows the advantage of implicit methods in terms of stability, however, for large dimensional problems and especially the nonlinear and ill-conditioned ones, implicit methods become very costly and difficult to implement. Here appears the need to a compromise between classical explicit methods and implicit integrators.

This compromise is "explicit stabilized methods" (see the survey \cite{Abd13c}). The idea is to construct explicit Runge--Kutta integrators with extended stability domain that grows quadratically with the number of stages $s$ of the method along the negative real axis, and then allows to use large time steps typically for problems arising from \emph{diffusion dominant} advection--diffusion--reaction PDEs for which the eigenvalues are close to the negative real axis and are very large in modulus. 

Before proceeding, let us recall some useful facts on Chebyshev polynomials. The first kind Chebyshev polynomials are defined by
\begin{equation}\label{eq:chebT}
	T_0(x)=1,\quad T_1(x)=x, \quad T_j(x)=2xT_{j-1}(x)-T_{j-2}(x),\quad j\geq2.
\end{equation}
The second kind Chebyshev polynomials are defined by
\begin{equation}\label{eq:chebU}
	U_0(x)=1,\quad U_1(x)=2x, \quad U_j(x)=2xU_{j-1}(x)-U_{j-2}(x),\quad j\geq2.
\end{equation}
Moreover, the two kinds polynomials satisfy the following
\begin{equation}\label{eq:TU}
	U_{j-1}(x)=\frac{T_j'(x)}{j}.
\end{equation}
The stabilization procedure is based on the above relations. The fact that both kinds share the same recurrence relation will be very useful in our analysis. {Indeed, this allows to simultaneously derive the recurrence formulas of the methods, otherwise, the cost would be doubled.}

{
	It was shown that for any explicit, consistent (order $1$) Runge--Kutta method, the maximum stability domain length over the negative real axis is $2s^2$, where $s$ is the number of stages of the method. The polynomial that achieves this \emph{optimal} length is the shifted Chebyshev polynomial $T_s(1+z/s^2)$. See, for example,  \cite[Chap.\ts V, Th.\ts 1.1]{HV03}. For robustness reasons, a damping of this polynomial is introduced and the resulting scheme is recalled in the following subsection.
}

\subsection{Optimal first order Chebyshev methods}
Consider the ODE
\begin{equation}\label{eq:ode}
	\dot y=f(y), \qquad y(0)=y_0,\qquad t\in[0,T].
\end{equation}
Given $y_0$, in order to compute $y_1\approx y(h)$ using the optimal first order Chebyshev method applied to \eqref{eq:ode} with step size $h$, the following recurrence is applied
\begin{eqnarray} 
	\label{eq:Cheb1} 
	K_0 &=& y_0, \quad
	K_1 = K_0+ \mu_1 h f(K_0), \nonumber \\
	K_j &=&\mu_ihf(K_{j-1})+\nu_iK_{j-1}+(1-\nu_i)K_{j-2},\quad j=2,\ldots,s\\
	y_{1} &=& K_s \nonumber,
\end{eqnarray}
where $\omega_{0}:=1+\frac{\eta}{s^{2}},\, \omega_{1}:=\frac{T_{s}(\omega_{0})}{T'_{s}(\omega_{0})}$, and
\begin{equation}\label{eq:coeffscheb}
	\mu_1:=\frac{\omega_1}{\omega_0},\quad
	\mu_j:=\frac{2\omega_1T_{j-1}(\omega_0)}{T_j(\omega_0)},\quad 
	\nu_j:=\frac{2\omega_0T_{j-1}(\omega_0)}{T_j(\omega_0)},\quad j=2,\ldots,s.
\end{equation}
The parameter $\eta$ is called the damping parameter and it is necessary to avoid singularities in the stability domain which ensures the robustness of the method {(See Figure \ref{fig:stabdomcheb})}. Typically for this method, $\eta$ is fixed to $0.05$. It can be easily verified, using the recurrence \eqref{eq:chebT} and proceeding by induction, that applied to the test problem \eqref{eq:odetest}, the above method produces after one step $y_1=K_s=R_s(h\lambda)y_0$ with
\begin{equation*}
	R_s(z) =\frac{T_{s}(\omega_0+\omega_1 z)}{T_{s}(\omega_0)},
\end{equation*}
for which the stability domain contains a narrow strip around the interval $[-C_\eta s^2,0]$ with  $${C_\eta=\frac{1+\omega_0}{s^2\omega_1}}\simeq 2-4/3\eta$$ is very close to $2$ (the optimal value for order 1). {For $\eta=0$ the stability function is again $T_s(1+z/s^2)$.} The method has low memory requirements (only two stages have to be stored) and reasonable propagation of round-off errors even for large values of $s$ needed in practice {\cite{HoS80,VHS90}.} 
The fact that the length of the stability domain on the negative  real axis enjoys a quadratic growth with respect to the number of stages $s$ is crucial to the success of explicit stabilized Runge--Kutta methods.

\begin{figure}
		\centering
		\includegraphics[width=0.9\linewidth]{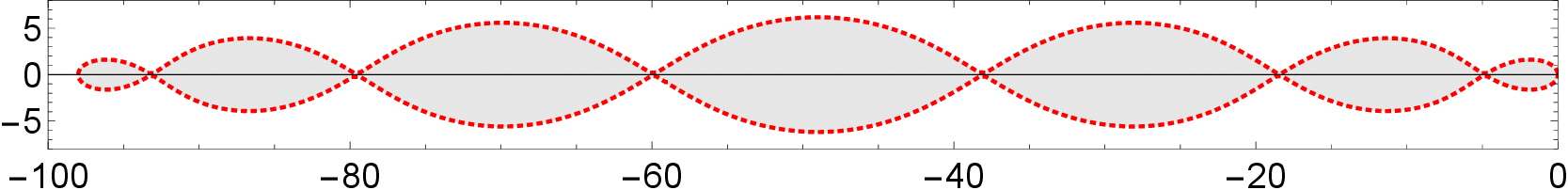}\\[2.ex]
		\includegraphics[width=0.9\linewidth]{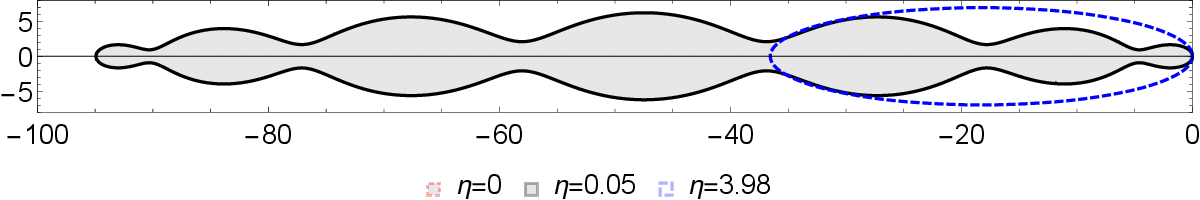}
	\caption{
		{Stability domains of the Chebyshev method \eqref{eq:Cheb1} for $s=7$ and different damping values $\eta=0,0.05,3.98$.}
		\label{fig:stabdomcheb}}
\end{figure}

\subsection{Second order RKC methods}

To design a second order method, we need the stability polynomial to satisfy\footnote{Indeed, up to order two, the order conditions for nonlinear problems are the same as the order conditions for linear problems \cite[Chap.\ts III]{HLW06}.}
\begin{equation*}
	R(z)=1+z+\frac{z^2}2+\mathcal{O}(z^3).
\end{equation*}
{A correction to the first order shifted Chebyshev polynomials was introduced to construct a stabilized scheme of order 2 \cite{B71,VHS80,VHS90}. The obtained second order polynomial is the following}
\begin{equation*}
	R_s(z) = a_s+b_sT_s(\omega_0+\omega_2z),
\end{equation*}
where,
\begin{equation}\label{eq:omegarkc}
	a_s=1-b_sT_s(\omega_0),\quad
	b_s=\frac{T''_s(\omega_0)}{(T'_s(\omega_0)^2)},\quad
	\omega_0=1+\frac{\eta}{s^2},\quad
	\omega_2=\frac{T'_s(\omega_0)}{T''_s(\omega_0)},\quad
	\eta=0.15.
\end{equation}
For each $s$, $|R_s(z)|$ remains bounded by $a_s+b_s=1-\eta/3 + \bigo(\eta^2)$ for $z$ in the stability interval (except for a small interval near the origin). The stability interval along the negative real axis is $[-\frac{1+\omega_0}{\omega_2},0]$ which is approximately $[-0.65s^2,0]$, and covers about $80\%$ of the optimal stability interval for second order stability polynomials, and the formula now for calculating $s$ for a given time step $h$ is
\begin{equation*}
	s := \left[\sqrt{\frac{h\lambda_{\max}+1.5}{0.65}}+0.5\right],
\end{equation*}
{where the brackets mean rounding to the nearest integer, and $\lambda_{\max}$ is the spectral radius of  the Jacobian of $f$ that can be calculated at each step using power method for example}.
Using the recurrence relation of the Chebyshev polynomials, the RKC method as introduced in {\cite{VHS90}} is defined by
\begin{equation} 
	\label{eq:rkc}
	\begin{split}
		K_0 &= y_0, \quad
		K_1 = K_0+h b_1\omega_2 f(K_0), \\
		K_j &=\mu_jh(f(K_{j-1})-a_{j-1}f(K_0))+\nu_jK_{j-1}+\kappa_jK_{j-2}+(1-\nu_j-\kappa_j)K_0, \\
		y_{1} &=K_s ,
	\end{split}
\end{equation}
where
\begin{equation}\label{eq:coeffsrkc}
	\mu_j=\frac{2b_j\omega_2}{b_{j-1}},~~ 
	\nu_i=\frac{2b_j\omega_0}{b_{j-1}},~~ 
	\kappa_j=-\frac{b_j}{b_{j-2}},~~
	b_j=\frac{T''_j(\omega_0)}{T'_j(\omega_0)^2},~~
	a_j=1-b_jT_j(\omega_0),
\end{equation}
for $j=2,\ldots,s$. The parameters $b_0$ and $b_1$ are free ($R_0(z)$ is constant and only order 1 is possible for $R_1$(z)) and the values $b_0=b_1=b_2$ are suggested in \cite{SV80}.

\begin{remark}
	For simplicity of the presentation, we consider the case of autonomous problems (with $f$ independent of time) but we highlight that our approach also applies straightforwardly to non-autonomous problems $\dot y(t)=f(t,y(t))$. Indeed, a standard approach is to consider the augmented system with $z(t)=t$, i.e. $\frac{dz}{dt}=1,~z(0)=0$ and define $\tilde{y}(t)=(y(t),z(t))^T$, see e.g. \cite[Chap.\ts III]{HLW06} for details.
\end{remark}

\section{Optimal first order scheme}\label{sec:o1}

Note that we have discussed the content of this section in the thesis \cite{A20}, but we recall it here since it gives insight about the construction of the second order adaptive integrator in the next section.

Consider the linear test problem
\begin{equation}\label{eq:adtest}
	\dot y=\lambda y+i\mu y,\qquad y(0)=y_0,
\end{equation}
where $\lambda \in \R^-$, $\mu\in\R$, and $i=\sqrt{-1}$. Applying a Runge--Kutta method to the above equation, one gets an induction of the form
\begin{equation*}
	y_{n+1}=R(p,q)y_n,
\end{equation*}
with $p=h\lambda$ and $q=h\mu$. We define the stability domain of a Runge--Kutta method applied to \eqref{eq:adtest} by
\begin{equation*}
	\calS=\{(p,q)\in \R^2~;~|R(p,q)|\leq 1\}.
\end{equation*}
Equation \eqref{eq:adtest} can be seen as the test equation for linear SDEs with the $i\mu$ replacing the noise. Hence, inspired by SK-ROCK \cite{AAV18}, we consider the following stability polynomial
\begin{equation}\label{eq:stabfundaeq}
	R(p,q) = A(p)+B(p)iq := \frac{T_s(\omega_0+\omega_1 p)}{T_s(\omega_0)} 
	+ \frac{U_{s-1}(\omega_0+\omega_1 p)}{U_{s-1}(\omega_0)}(1+ \frac{\omega_1}2 p) iq,
\end{equation}
where $T_s$ and $U_s$ are the first and the second kind Chebyshev polynomials of degree $s$ (the number of stages), and the coefficients $\omega_0$ and $\omega_1$ are the same as for the Chebyshev method \eqref{eq:Cheb1}.
The stability condition $|R(p,q)|\leq 1$ is equivalent to $A(p)^2+B(p)^2q^2\leq1$ which is exactly the mean square stability condition described in \cite[Sect. 2]{AAV18}. By \cite[Theorem 3.2]{AAV18} and \cite[Remark 3.6]{AAV18}, for $\eta>0$ and $s\in\IN$, $|R(p,q)|\leq1$ for all $p\in [-2\omega_1^{-1},0]$ and all $q$ such that $|q|\leq\sqrt{-2p}$ (See Figure \ref{fig:stabadvdiff}).

{
\begin{remark}
	For SDEs, the condition $|q|\leq\sqrt{-2p}$ guarantees the stability of the solution of the continuous problem, which means that numerical stability under this condition is sufficient. However, for our test equation \eqref{eq:adtest}, the exact solution is stable when $\lambda\leq0$ and $\mu\in\R$, hence even when $|q|\leq\sqrt{-2p}$ is violated. Although stability under this condition is already a significant improvement comparing to standard explicit methods, we can use larger damping parameter $\eta$ to increase the width of the stability region along the imaginary direction in the case of large advection.
\end{remark}
}


\begin{figure}[t]
	\centering
	\includegraphics[width=0.9\linewidth]{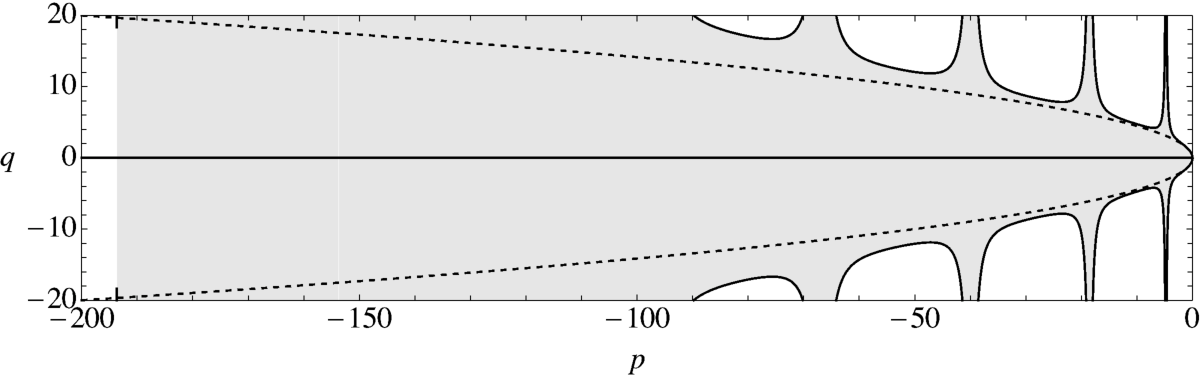}
	\caption[Stability domain of the new first order method]{Stability domain of the new optimal first order method \eqref{eq:adc} in the $p-q$ plane for $s=10$ and $\eta=0.05$. The dashed lines correspond to $\pm\sqrt{-2p}$.}
	\label{fig:stabadvdiff}
\end{figure}

The new order one method for space discretized advection--diffusion--reaction equations \eqref{eq:odead}, is defined in the same way as the SK-ROCK method \cite{AAV18}, by replacing the noise term by the advection-reactions terms,
\begin{equation}\label{eq:adc}
	\begin{split}
		K_0&=y_0 \\
		K_1&=y_0+\mu_1hF_D(y_0+\nu_1 hF_A(y_0)) +\kappa_1hF_A(y_0)\\
		K_j&=\mu_jhF_D(K_{j-1})+\nu_jK_{j-1}+\kappa_jK_{j-2},\quad j=2,\dots,s,\\
		y_1&=K_s,
	\end{split}
\end{equation}
where $\nu_1=s\omega_1/2$, $\kappa_1=s\omega_{1}/\omega_{0}$ and the rest of the coefficients are identical to those defined in \eqref{eq:coeffscheb}. Assuming enough regularity on $F_D$ and $F_A$, the convergence proof is straightforward and based on \cite[Lemma 4.2]{AAV18}. {The above scheme is also optimal in the sense that its stability region achieves the maximum possible length over the negative real axis for an explicit consistent scheme.}

Note that the method requires only 1 evaluation of the advection-reaction terms $F_A$ per time step, while other existing methods for advection diffusion problems such as PIROCK \cite{AbV12a} and PRKC \cite{Zb11} require 3 and 4 evaluations of $F_A$ respectively per time step. It also outperform the mentioned methods by far in terms of stability. However, it has only order 1 of accuracy, which motivates the construction of a second order integrator that competes with existing schemes in terms of stability and convergence, taking advantage of the above analysis.

Note that we can adaptively increase the stability for advection by increasing $\eta$ in the price of loosing some length on the negative real axis. This will be explained in details in the next section for the second order scheme, for which this adaptive increase of damping is a key feature.
\section{New second order scheme} \label{sec:o2}
In this section, we introduce the new adaptive second order integrator based on the material presented in the first 3 sections and inspired by the SK-ROCK method introduced in \cite{AAV18}.
\paragraph{ARKC integrator} The ARKC integrator (for adaptive RKC) applied to \eqref{eq:odead} is defined as follows for $s\geq2$:
\begin{equation}\label{eq:ARKC}
	\begin{split}
		G&=hF_A(y_0+\frac h2F_A(y_0+\frac{\omega_2}2hF_D(y_0))+\frac h2F_D(y_0))+hF_D(y_0+\frac{\omega_2-1}2hF_A(y_0))\\
		&-hF_D(y_0),\\
		K_{-1}&=y_0,\\	
		K_0&=K_{-1}+\frac{\omega_2}2G,\\
		K_1&=K_0+b_1\omega_2hF_D(y_0)+\alpha G, \quad \text{and for } j=2,\dots,s,\\
		K_j&=\mu_jh(F_D(K_{j-1})-F_D(K_0)+(1-a_{j-1})F_D(y_0))+\nu_jK_{j-1}+\kappa_jK_{j-2}\\
		&+(1-\nu_j-\kappa_j)K_0,\\
		y_{1}&=K_s,
	\end{split}
\end{equation}
where $\alpha=\left(1-\frac{\omega_2}2\right)b_1s\omega_2,$
and the other coefficients are identical to those defined in \eqref{eq:omegarkc} and~\eqref{eq:coeffsrkc}. Note that for purely diffusive equations ($F_A\equiv0$), the method reduces to standard RKC \eqref{eq:rkc}.

\paragraph{Complexity} The method requires $s+2$ evaluations of $F_D$ and $3$ evaluations of $F_A$ per time step. The standard RKC and the PRKC methods both require $s$ evaluations of $F_D$ per time step, but RKC requires $s$ evaluations of $F_A$ while PRKC requires only $4$. The PIROCK merhod needs $s+2+l$ evaluations of $F_D$ ($l=1$ or $2$) and $3$ evaluations of $F_A$ per time step. {Finally, the IMPRKC scheme requires $s+\hat s$ evaluations of each $F_D$ and $F_A$, where $\hat s$ becomes quite large for very stiff problems.} The advantage of our new scheme comes from having larger stability region that changes adaptively, which allows to use bigger time steps.

\paragraph{Construction} In order to construct a second order scheme, we need our stability function to satisfy the following equality
\begin{equation}\label{eq:o2p}
	\begin{split}
		R_2(p,q)&=1+p+iq+\frac12(p+iq)^2+\bigo(((p+iq)^3)\\&=1+p+iq+\frac{p^2}2+ipq-\frac{q^2}2+\bigo((p+iq)^3).
	\end{split}
\end{equation}
{A natural approach would be to modify the last stage of the first order scheme to reach second order, but such naive modification will cause severe instability. Our idea is to start the stabilization using second kind Chebyshev polynomials from the beginning of the integration process.} Hence, inspired by the previous section, we consider the following polynomial
\begin{equation}\label{eq:stabfun2}
	\begin{split}
		R_2(p,q) &= A_2(p)+B_2(p)(iq-\frac{q^2}2)\\
		& := a_s+b_sT_s(\omega_0+\omega_2 p) \\
		&+
		\left(\frac{\omega_2}2+\left(1-\frac{\omega_2}2\right)\frac{U_{s-1}(\omega_0+\omega_2 p)}{U_{s-1}(\omega_0)}\right)(1+ \frac{\omega_2}2 p) (iq-\frac{q^2}2),
	\end{split}
\end{equation}
where all the coefficients are defined in \eqref{eq:omegarkc}, from which we can derive the new adaptive second order ARKC method \eqref{eq:ARKC} for advection--diffusion--reaction problems, using the relations \eqref{eq:chebT}, \eqref{eq:chebU}, and \eqref{eq:TU}. {The first part of our stability function \eqref{eq:stabfun2} corresponds to the stability polynomial of the standard second order RKC scheme, while the second part involves second kind Chebyshev polynomials (as in \eqref{eq:stabfundaeq}) in order to stabilize the advection part, and has the correct order. The construction of the method from the polynomial \eqref{eq:stabfun2} is done by induction using the relations \eqref{eq:chebT}-\eqref{eq:TU}. See also Lemma \ref{lemma:stab} below.}

Unlike the case of standard ODEs where the term $z^2/2$ is enough to have order two for nonlinear problems, some additional \emph{coupling} conditions need to be satisfied here due to the partitioned nature of the scheme (see \cite[Sect.\ts III.2]{HLW06}). These conditions result from the fact that the term $ipq$ in \eqref{eq:o2p} is in fact the sum of two terms: $\frac12 ipq$ and $\frac12 iqp$, which are not necessarily equal for multidimensional and/or nonlinear problems.
{The quantity $G$ is constructed carefully to get order $2$ of convergence for general nonlinear problems. For linear equations, $G=(1+\frac{\omega_2}2)(iq-\frac{q^2}2)$.}
Analogously to some existing second order methods \cite{AbV12a,VSH04}, our damping parameter is not fixed to a small value, but it is an increasing function of the stage number $s$. However, for relatively large Peclet numbers, no huge damping is needed, and the method still perform very well as will be shown in Section \ref{sec:num}.

\subsection{Stability analysis}
\label{sec:stab}

{
	\begin{remark}
The stability analysis of Runge--Kutta type methods is usually made on linear test problems, but we emphasize that such test equations only give insight on the stability of the stabilized method under study (here ARKC). This is because practical problems are usually nonlinear, and even for linear problems, the involved operators cannot in general be diagonalized simultaneously. 
\end{remark}
}
\begin{lemma}\label{lemma:stab}
	The scheme \eqref{eq:ARKC} applied to the linear test problem \eqref{eq:adtest}  with step size $h$, produces the following recurrence
	\begin{equation*}
		y_{n+1}=R_s(h\lambda,h\mu)y_n,
	\end{equation*}
	where the stability polynomial $R_s(p,q)$ is defined in \eqref{eq:stabfun2}.
\end{lemma}

\begin{proof}
	By induction on $j$, it can be shown that, for every $j\geq1$, the internal stages satisfy 
	
	$$K_j=R_j(p,q)y_n,$$
	where,
	\begin{equation*}
		\begin{split}
			R_j(p,q)&\coloneqq a_j+b_jT_j(\omega_0+\omega_2 p) \\
			&+\left(\frac{\omega_2}2+\left(1-\frac{\omega_2}2\right)U_{j-1}(\omega_0+\omega_2 p)\right)b_js\omega_2(1+ \frac{\omega_2}2 p) (iq-\frac{q^2}2).
		\end{split}
	\end{equation*}
	For $j=s$, we have that $b_ss\omega_2=1/U_{s-1}(\omega_{0})$ because $s^{-1}T_s'(\omega_{0})=U_{s-1}(\omega_{0})$ and the proof is done. \qed	
\end{proof}

In contrast to the first order scheme \eqref{eq:adc} and the stochastic integrator SK-ROCK \cite{AAV18}, the damping parameter for the ARKC method is not fixed. Indeed, it is an increasing function of the Peclet number and the number of stages $s$. We provide numerical stability analysis that illustrates the nice features of the new scheme (see Figures \ref{fig:stabreg} and \ref{fig:stabpol}).

\begin{figure}[tb]
	\centering
	\begin{subfigure}[t]{0.45\textwidth}
		\includegraphics[width=1\linewidth]{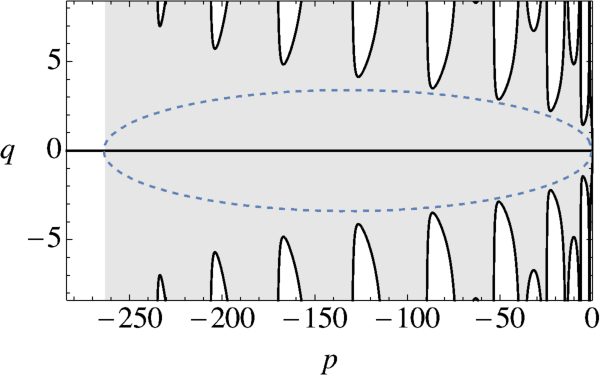}
		\caption{$\eta=0.15$, $\delta_s=0.65s^2$, $\alpha_s=0.17s$.}
					\label{fig:eta015}
	\end{subfigure}
	\begin{subfigure}[t]{0.45\textwidth}
		\includegraphics[width=1\linewidth]{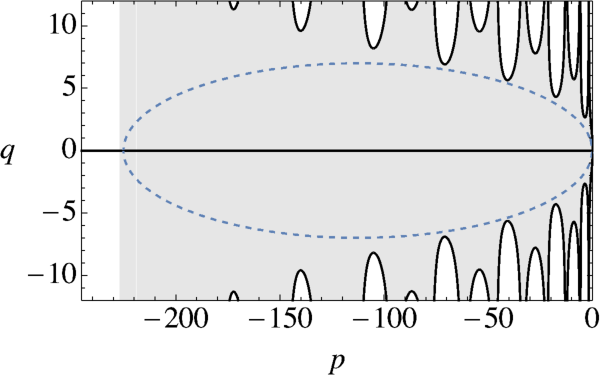}
		\caption{$\eta=1.5$, $\delta_s=0.56s^2$, $\alpha_s=0.35s$.}
	\end{subfigure}\\
	\begin{subfigure}[t]{0.45\textwidth}
		\includegraphics[width=1\linewidth]{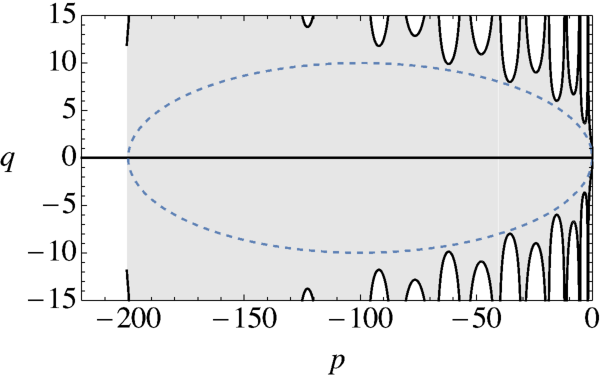}
		\caption{$\eta=3$, $\delta_s=0.5s^2$, $\alpha_s=0.5s$.}
					\label{fig:eta3}
	\end{subfigure}
	\begin{subfigure}[t]{0.45\textwidth}
		\includegraphics[width=1\linewidth]{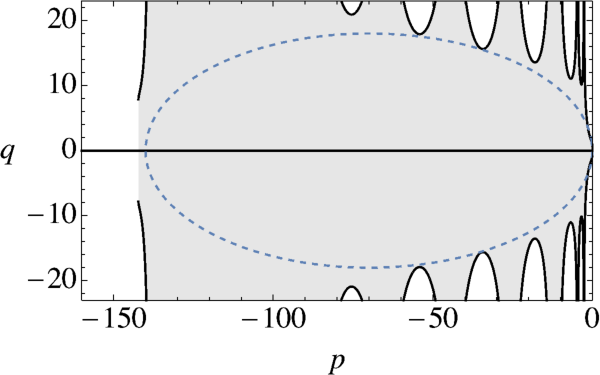}
		\caption{$\eta=10$, $\delta_s=0.35s^2$, $\alpha_s=0.9s$.}
					\label{fig:eta10}
	\end{subfigure}
	
	\caption{Illustration of the stability regions of the ARKC method \eqref{eq:ARKC} in the $p-q$ plane for $s=20$ and different values of the damping parameter $\eta$. {For each stability region, we plot the largest possible ellipse that fits inside.}}
	\label{fig:stabreg}
\end{figure}
In what follows, we will respectively denote by {$\delta_s$ and $\alpha_s$} the half width and the half height of the largest ellipse that can be put in the stability region of the corresponding method.

The method \eqref{eq:ARKC} is designed to handle quite large Peclet numbers, this means that we need the width of the stability region in the imaginary direction to be as large as possible. While the length over the negative real axis grows quadratically with $s$, we cannot achieve more than linear growth on the imaginary direction (this fact is given as an exercise in \cite[Chap.\ts IV]{HaW96}). In \cite{VSH04}, the RKC method is proposed with large damping value which reduces the length of the stability domain over the negative real axis to $\delta_s=0.34s^2$, and leads to a growth of only $\bigo(\sqrt{s})$ for the ellipse half-height $\alpha_s$. The PRKC method proposed in \cite{Zb11} uses the standard small damping $\eta=0.15$ that keeps $\delta_s\approx0.65s^2$, but $\alpha_s$ is fixed to $1.7$ no matter how large is the number of stages $s$, which does not add much to the standard RKC method since this might be useful only in the small Peclet number regime. However, the main feature of the method is to reduce the number of evaluations of possible non stiff terms such as non-stiff advection or reaction terms. The PIROCK integrator  \cite{AbV12a}, is notably better than the two mentioned methods RKC and PRKC. Two kinds of damping are proposed for PIROCK, the first lead to a nearly optimal length over the negative real axis $\delta_s\approx0.81s^2$ but $\alpha_s$ is limited to $ 0.07696s + 1.878$. For the second damping, $\alpha_s=0.5321s + 0.4996$ but $\delta_s$ reduces to $0.43s^2$. It shares as well the feature of reducing the evaluation of non diffusive terms. Nevertheless, PIROCK still have two drawbacks: the first one is the non-availability of explicit closed form formulas to compute its coefficients for a given number of stages $s$, the other is the limited choice of damping.

The stability region of the ARKC method changes adaptively with the spectrum of the Jacobian of the vector fields at each time step. This change is due to the damping parameter $\eta$ which is not necessarily constant. In fact, the ellipse half width $\delta_s$ still grows quadratically with $s$, while the ellipse half height $\alpha_s$ grows linearly with $s$, in contrast to RKC and PRKC methods. For example, for the standard value of $\eta=2/13$ used for RKC in \cite{SSV98} and PRKC in \cite{Zb11}, $\delta_s$ is still equal to $0.65s^2$ and $\alpha_s\simeq0.17s$ (see Figure \ref{fig:eta015}), whereas for PRKC $\alpha_s$ is fixed to $1.7$ .  In addition, increasing the value of the damping parameter $\eta$ adds more space in the imaginary direction which is very favorable for the advection dominated problems. Figure \ref{fig:eta3} illustrates the stability region in the p-q plane for $s=20$ and $\eta=3$. We can see that $\delta_s=0.5s^2$ and $\alpha_s=0.5s$, while the best that PIROCK \cite{AbV12a} could achieve for almost the same $\alpha_s$ is $\delta_s=0.43s^2$ which makes difference for large values of $s$. For $\eta=10$ we have $\delta_s=0.35s^2$ and $\alpha_s=0.9s$ as shown in Figure \ref{fig:eta10}, compare that with the case of standard RKC with infinite damping, considered in \cite{VSH04}, where $\delta_s=0.34s^2$ while $\alpha_s=\bigo(\sqrt{s})$. {The last method to compare with is IMPRKC introduced in \cite{TX20} where the authors provide a plot for $s=50$ stages and $\hat s=3$ additional stages. The length over the negative real axis stays almost the same as RKC at $0.65s^2$, while the width along the imaginary axis is around $0.8s$. As said before, the evaluation of the advection--reaction term at every step of IMPRKC is a disadvantage that limits its performance. For instance, for small Peclet numbers it needs time steps as much as the standard RKC scheme which makes it a bit more expensive because of the additional stages needed. In Section \ref{sec:num}, our scheme ARKC is shown to perform better in all regimes.}

\subsection{Choice of damping} \label{sec:cod}
In fact, $\alpha_s$ is a function of $s$ and $\eta$, this introduces additional difficulty in the estimation of the value of $\alpha_s$. Therefore, we will introduce many choices of the range of Peclet number to simplify the implementation.  Let $\rho_D$ and $\rho_A$ be the spectral radii of the Jacobians of $F_D$ and $F_A$ respectively.
{For ODEs coming from the discretization of the linear PDE \eqref{eq:ADlin}, we have $\rho_D = 4d/\Delta x^2$ and $\rho_A= a/\Delta x$, then
$$
\frac{\rho_A}{\sqrt{\rho_D}}=\frac{a}{2\sqrt{d}}.
$$
We define the number $$P_e'\coloneqq\frac{a}{\sqrt{d}},$$
hence, $\rho_A/\sqrt{\rho_D}=\frac12P_e'$. Similarly, we can show that $q=\frac12P_e'\sqrt{-hp}$. We will use this new number $P_e'$ to adapt our choice of damping according to the parameters of the problem. Notice that this is just another measure of the advection dominance that is proportional to the Peclet number and more convenient to use in our algorithm.}

 {For  $P_e'\leq1/10$, i.e, $\rho_A/\sqrt{\rho_D}\leq1/20$, we fix  $\eta=0.15$ up to $s=200$, and for $s$ between $200$ and $500$ we set $\eta=0.6$ (we do not allow $s$ to be more than $500$). 
For $1/10<P_e'\leq1/2$ which means that $1/20<\rho_A/\sqrt{\rho_D}\leq1/4$,} we consider the following choices for $s$ and $\eta$:
\begin{center}
	
	\begin{tabular}{ |c|c|c|c|c|c|c|c|c|c|c|c|} 
		\hline
		$2\leq s\leq30$& $31\leq s\leq60$ & $61\leq s\leq110$ & $111\leq s\leq 160$\\
		\hline
		$\eta= 0.2$ & $ \eta=0.45$ &$\eta=1$&$\eta=1.5$\\
		\hline
	\end{tabular}
	\begin{tabular}{ |c|c|c|c|c|c|c|c|c|c|c|c|} 
	\hline
	$161\leq s\leq260$ & $261\leq s\leq360$ & $361\leq s\leq500$\\
	\hline
	$\eta=2.4$&$\eta=3$&$\eta=4$\\
	\hline
\end{tabular}
	\captionof{table}{$1/10<P_e'\leq1/2,\,1/20<\rho_A/\sqrt{\rho_D}\leq1/4$.}
	\label{table0p5}
\end{center}
Tables for other values of $P_e'$ are given in appendix \ref{app}. Infinitely many choices could be done, but we will present and use only a few choices since they are enough for the method to perform very well. The methodology to compute the above values is easy, it is enough to plot $R_s(p,q)$ for the corresponding choice of $q$ and to vary $\eta$ in a way that we stay stable for the given value of $s$ (see Figure \ref{fig:stabpol}).

\begin{figure}[tb]
	\centering
	\begin{subfigure}[t]{0.32\textwidth}
		\includegraphics[width=1\linewidth]{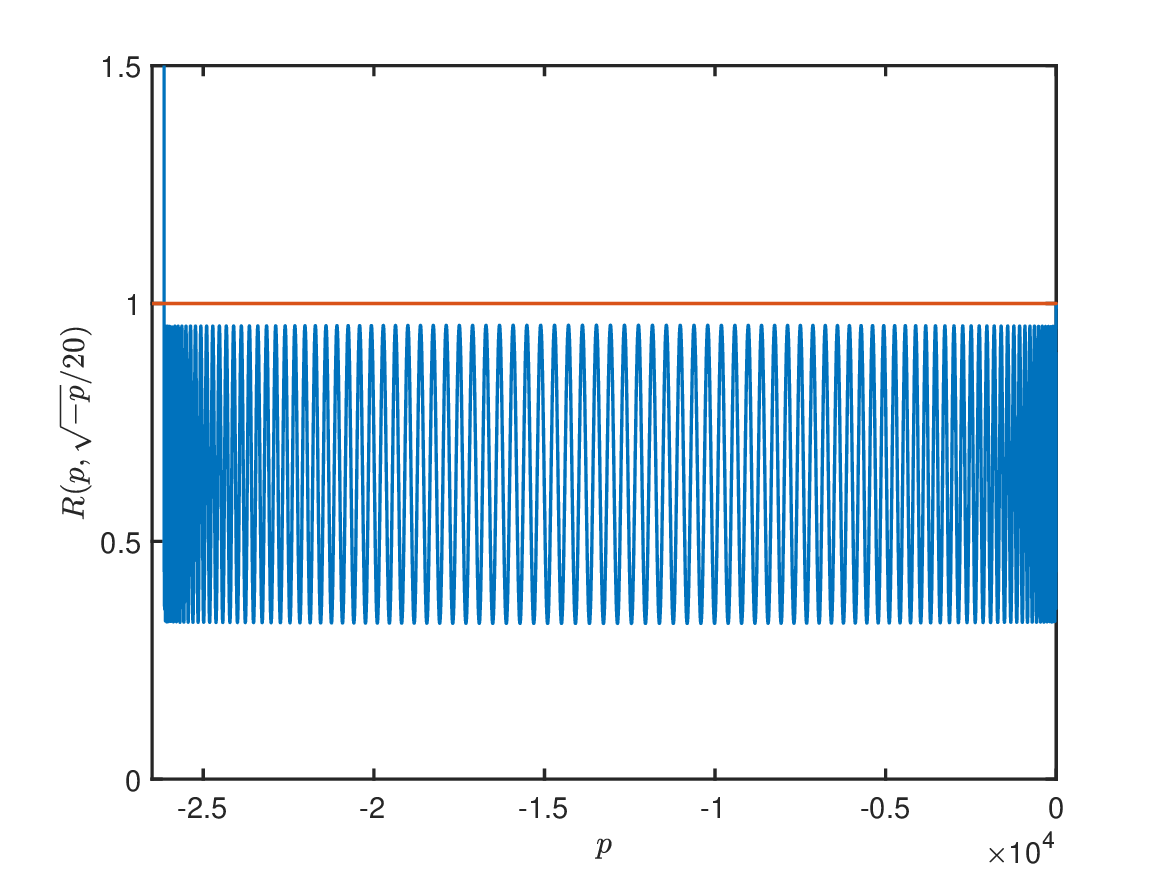}
		\caption{$\eta=0.15$, $\delta_s=0.65s^2$.}
	\end{subfigure}
	\begin{subfigure}[t]{0.32\textwidth}
		\includegraphics[width=1\linewidth]{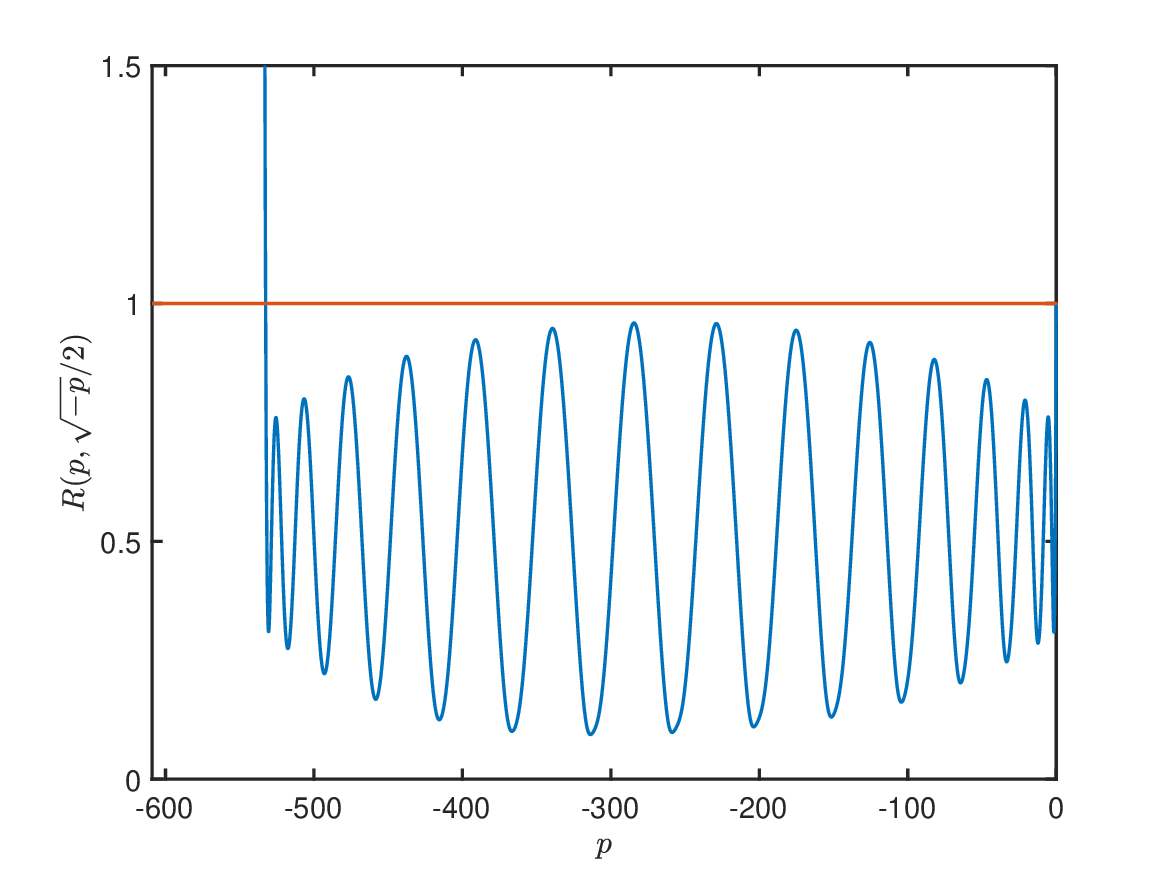}
		\caption{$\eta=1$, $\delta_s=0.59s^2$.}
	\end{subfigure}
	\begin{subfigure}[t]{0.32\textwidth}
		\includegraphics[width=1\linewidth]{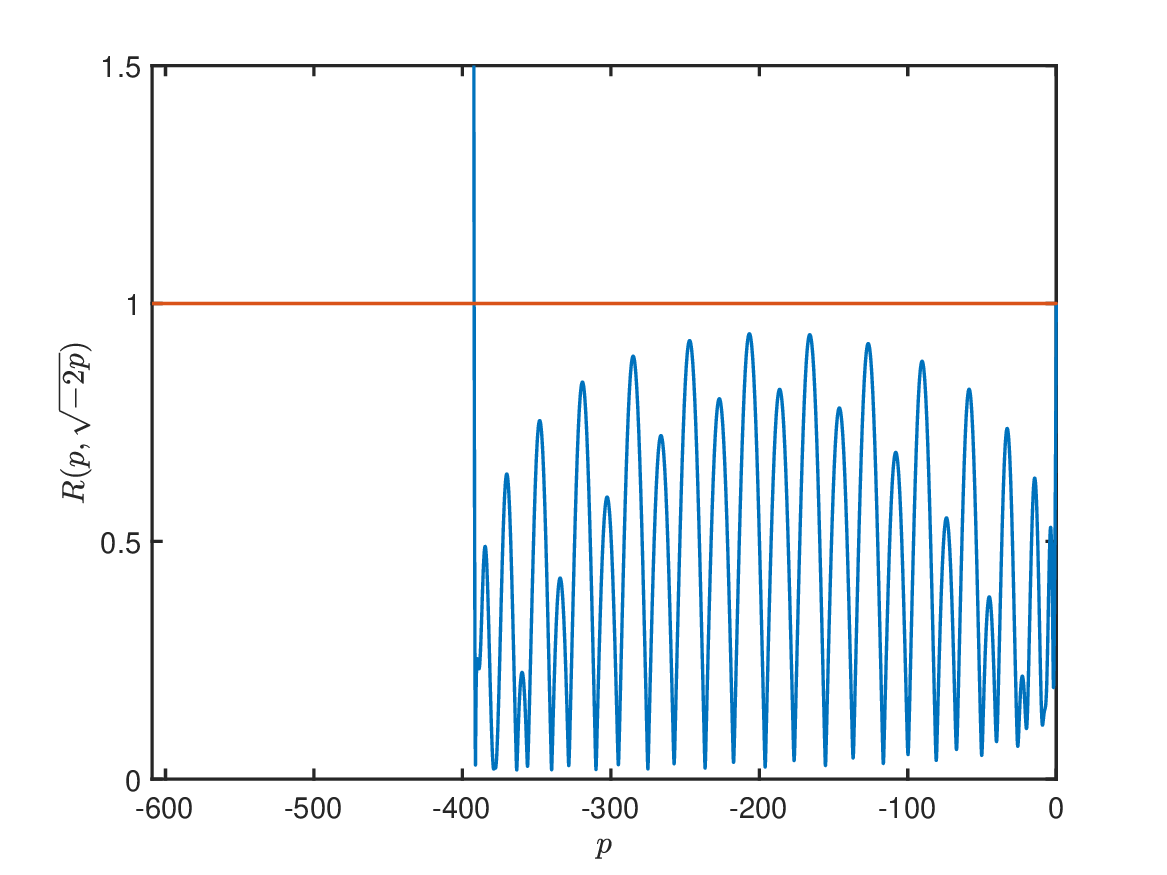}
		\caption{$\eta=5$, $\delta_s=0.4345s^2$.}
	\end{subfigure}
	\caption{Stability polynomial \eqref{eq:stabfun2} for $s=200$ (left) and $s=30$ (middle and right), and different values of $\eta$ and $q$.}
	\label{fig:stabpol}
\end{figure}

\subsection{Convergence analysis}
In this section we will prove the second order of convergence of the scheme \eqref{eq:ARKC} when applied to ODEs of the form \eqref{eq:odead} arising from the discretization of advection--diffusion--reaction problems.
\begin{theorem}
	Let $T>0$ and consider the system of ODEs \eqref{eq:odead}  on the time interval $[0,T]$, where $F_D$ and $F_A$ are of class $\mathcal{C}^2$ and are Lipschitz continuous. Suppose in addition that the first and second derivatives of $F_D$ and $F_A$ are bounded. Let $N\in\IN$, $h=T/N$, and $t_n=nh,~n=1,\dots,N$, {and consider the method \eqref{eq:ARKC} applied to \eqref{eq:odead} with step size $h$, such that the number of stages $s$ and the damping parameter $\eta$ are chosen appropriately to guarantee stability.} Then, we have for all $n=1\dots,N$,
	\begin{equation}\label{eq:global}
		\|y(t_n)-y_n\| \leq Ch^2,
	\end{equation}
	where $C$ is independent of $h$ and $n$. {In other words, the method converges with order 2.}
\end{theorem}

\begin{proof}
	Let us prove first that the local error (the error after one step) satisfies 		
	\begin{equation}\label{eq:local}
		\|y(h)-y_1\| =\bigo(h^3).
	\end{equation}
	Throughout the proof, $F_i(y_0)$ will be simply denoted by $F_i$, and $F_i'(y_0)F_j(y_0)$ will be denoted by $F_i'F_j$ with $i,j\in\{D,A\}$.
	
	Using Taylor expansion, we can easily see that 
	$$
	G=hF_A+\frac{h^2}2F_A'F_A+\frac{h^2}2F_A'F_D+(\omega_2-1)\frac{h^2}2F_D'F_A+\bigo(h^3).
	$$
	Now, let us suppose that
	\begin{equation}\label{eq:stage}
		K_j=y_0+\gamma_1^jhF_D+\gamma_2^jhF_A+\gamma_3^jh^2F_D'F_D+\gamma_4^jh^2F_D'F_A+\gamma_5^jh^2F_A'F_D+\gamma_6^jh^2F_A'F_A+\bigo(h^3).
	\end{equation}
	The first two stages of the method satisfy:
	\begin{equation*}
		K_0=y_0+\frac{\omega_2}2G,\quad K_1=y_0+b_1\omega_2hF_D(y_0)+(\alpha+\frac{\omega_2}2)G,
	\end{equation*} 
	hence, we have,  $\gamma_1^0=0$, $\gamma_1^1=b_1\omega_2$, $\gamma_2^0=\frac{\omega_2}2$, $\gamma_2^1=\alpha+\frac{\omega_2}2$,
	$\gamma_3^0=\gamma_3^1=0$,
	$\gamma_4^0=\frac{\omega_2(\omega_2-1)}4$, $\gamma_4^1=(\alpha+\frac{\omega_2}2)\frac{\omega_2-1}2$,
	$\gamma_5^0=\frac{\omega_2}4$, $\gamma_5^1=\frac12(\alpha+\frac{\omega_2}2)$,
	$\gamma_6^0=\frac{\omega_2}4$, and finally, $\gamma_6^1=\frac12(\alpha+\frac{\omega_2}2)$. By performing a Taylor expansion of the stages $K_j$ defined in \eqref{eq:ARKC} and replacing $K_{j-1}$ and $K_{j-2}$ by the expansion defined in\eqref{eq:stage}, and finally identifying the coefficients, we get the following relations
	\begin{equation*}
		\begin{split}
			\gamma_1^j&=\mu_j(1-a_{j-1})+\nu_j\gamma_1^{j-1}+\kappa_j\gamma_1^{j-2},\\
			\gamma_2^j&=\nu_j\gamma_2^{j-1}+\kappa_j\gamma_2^{j-2}+(1-\nu_j-\kappa_j)\frac{\omega_2}2,\\
			\gamma_3^j&=\mu_j\gamma_1^{j-1}+\nu_j\gamma_3^{j-1}+\kappa_j\gamma_3^{j-2},\\
			\gamma_4^j&=\mu_j(\gamma_2^{j-1}-\frac{\omega_2}2)+\nu_j\gamma_4^{j-1}+\kappa_j\gamma_4^{j-2}+(1-\nu_j-\kappa_j)\frac{\omega_2(\omega_2-1)}4,\\
			\gamma_5^j&=\nu_j\gamma_5^{j-1}+\kappa_j\gamma_5^{j-2}+(1-\nu_j-\kappa_j)\frac{\omega_2}4,\\
			\gamma_6^j&=\nu_j\gamma_6^{j-1}+\kappa_j\gamma_6^{j-2}+(1-\nu_j-\kappa_j)\frac{\omega_2}4.
		\end{split}
	\end{equation*}
	In order to prove \eqref{eq:local}, it is sufficient to show that $\gamma_1^s=\gamma_2^s=1$ and $\gamma_3^s=\gamma_4^s=\gamma_5^s=\frac12$. Obviously, $\gamma_6^j=\gamma_5^j$ for all $j\geq0$. We will provide proofs for the first 3 coefficients, the other two can be done using similar arguments.
	\begin{itemize}
		\item For $\gamma_1^j$ we have: $\gamma_1^0=0$, $\gamma_1^1=b_1\omega_2$, and $\gamma_1^j=\mu_j(1-a_{j-1})+\nu_j\gamma_1^{j-1}+\kappa_j\gamma_1^{j-2}$, thus $\gamma_1^j$ are the internal stages of the RKC method $\eqref{eq:rkc}$ applied to the problem
		$
		\dot y=1,~ y(0)=0
		$
		with step size $h=1$. This means that for all $0\leq j\leq s$, $\gamma_1^j=c_j$, where $c_j$ is the $j^{th}$ node of the RKC method. Therefore, $\gamma_1^s=c_s=1$. Moreover, according to \cite[Sect.\ts 2]{SSV98}, for all $j=2\dots,s$ we have 
		$$
		\gamma_1^j=c_j=\omega_2\frac{T_j''(\omega_0)}{T_j'(\omega_0)},
		\qquad c_1=\frac{c_2}{T_2'(\omega_0)}=b_1\omega_2,\quad c_0=0.
		$$
		
		\item It can be proved that $\gamma_2^j,~j=2,\dots,s$ are given by
		$
		\gamma_2^j=\frac{\omega_2}2+\alpha\frac{b_j}{b_1}P_j(\omega_0),
		$
		where $P_j(x)$ are polynomials that satisfy the following two term recurrence relation
		$$
		P_0(x)=0,\quad P_1(x)=1, \quad P_j(x)=2xP_{j-1}(x)-P_{j-2}(x),~j\geq 2.
		$$
		Comparing with the relation \eqref{eq:chebU}, it can be easily seen that for all $j\geq 1$, 
		$$P_j(x)=U_{j-1}(x)=T_j'(x)/j.$$ Hence, $\gamma_2^s=\frac{\omega_2}2+\alpha\frac{b_s}{sb_1}T_s'(\omega_0)=\frac{\omega_2}2+(1-\frac{\omega_2}2)\omega_2b_sT_s'(\omega_0)=1.$
		
		\item The proof for $\gamma_3^j$ is very similar to that of $\gamma_2^j$. Indeed, we can prove for each  $j=0,\dots,s$ the equality $\gamma_3^j=2\omega_2^2b_jQ_j(\omega_{0})$, where 
		$$
		Q_0(x)=0,\quad Q_1(x)=0, \quad Q_j(x)=T_{j-1}'+2xQ_{j-1}(x)-Q_{j-2}(x),~j\geq 2.
		$$
		Using the relation \eqref{eq:chebT}, we observe that for all $j \geq 0$, $Q_j(x)=\frac{T_j''(x)}4$, which implies that $\gamma_2^s=2\omega_2^2b_sT_s''(\omega_0)/4=1.$
	\end{itemize}
	Thus, \eqref{eq:local} is proved, and using regularity assumptions made on the vector fields, \cite[Theorem\ts 3.6]{HNW93} implies the global convergence estimate \eqref{eq:global}. \qed
		
\end{proof}

\subsection{Variable step size control and the fully adaptive algorithm}
We introduce the following local error estimator that allows us to adaptively select the time step size in order to reach a given accuracy,
\begin{equation*}
	Est_{n+1}=C(12(y_n-y_{n+1})+6h(F_D(y_n)+F_A(y_n)+F_D(y_{n+1})+F_A(y_{n+1}))),
\end{equation*}
where $C=1/6-c_2+(1/2-c_1)\zeta-1/6\zeta$, with $\zeta=0$ if $F_A \equiv 0$ and $\zeta=1$ otherwise, and 
$$
c_1=\frac{\omega_2}{2}\left(1-\frac{\omega_2}{2}\right)  \left(1+\omega_2\frac{U''_{s-1}(\omega_0)}{U_{s-1}(\omega_0)}\right),\quad c_2=sb_sU''_{s-1}(\omega_0)\frac{\omega_2^3}{6}.
$$
To get an intuition about the above coefficients, compare the third order term in the exact polynomial which is $(p+iq)^3/6=p^3/6-iq^3/6+iqp^2/2-pq^2/2$ and the third order term of stability polynomial \eqref{eq:stabfundaeq} that is equal to $c_2p^3+c_1iqp^2-pq^2/2$. Note that the above estimator is inspired by the one considered for RKC in the paper \cite{SSV98}, and they coincide for $\zeta=0$. In contrast to the estimators introduced for PRKC \cite{Zb11} and PIROCK \cite{AbV12a}, we do not consider two separate estimators for $F_D$ and $F_A$, since our method is defined in a different way. Indeed, for improved stabilization, the advection-reaction terms are computed at the beginning and not separately at the end, which makes the method more similar to RKC.
We adopt the standard step size selection strategy proposed in \cite{SSV98} for RKC (see also \cite[page 167]{HNW93}).
Now, we are ready to present our fully adaptive algorithm,
\begin{algo}[$y_0 \mapsto y_1$]\label{algo}
	Given a time step size $h$ {and an initial value $y_0$}:
	\begin{itemize}
		\item {Calculate $\rho_A$ and $\rho_D$ at the current value of the solution.}
		\item Calculate $\rho_A/\sqrt{\rho_D}$ and choose the corresponding table among Tables \ref{table0p5}, \ref{table1}, \ref{table1p5}, \ref{table2}, \ref{table2s2},~\ref{tablelarge}.
		\item Search the minimum $s$ (and the corresponding $\eta$) in the chosen table such that $\frac{1+\omega_0}{\omega_1}>h\rho_D$.
		\item Generate the coefficients \eqref{eq:omegarkc} and \eqref{eq:coeffsrkc} and apply the recurrence \eqref{eq:ARKC} to calculate $y_1$.
		\item Update the step size according to the automatic step size selection procedure and repeat until reaching the final time.
	\end{itemize}
\end{algo}

{
\begin{remark}
	To ensure stability, it is enough to choose $s$ such that $\frac{1+\omega_0}{\omega_1}>h\rho_D$, because the relation between $\eta$ and the corresponding range for $s$ in each table is built to ensure that, once $-h\rho_D$ lies inside the stability domain, the whole ellipse containing the eigenvalues for the given spectral radii fits inside.
\end{remark}
}

{
\begin{remark}
	In the case where the eigenvalues have nonzero imaginary part and very small real part, this is close to purely advective regime, which is out of the scope of the paper. However, this case can be treated in two different ways, either by using the adaptive algorithm without any modification and then the error estimator will choose a very small time step, or by modifying the algorithm to integrate such systems (extremely large Peclet number) with RK3 or RK4 (for which the stability domain includes a part of the imaginary axis).
\end{remark}
	}

{For the calculation of the spectral radii $\rho_D$ and $\rho_A$, we use the Matlab function "eig". In other programming languages, one can use nonlinear power method as in \cite{Abd02b} for example. However, the cost of such methods is a different issue and is beyond the scope of the present paper.} The code of the ARKC integrator as well as the drivers that reproduce the numerical experiments will be made publicly available on the page: \href{https://sites.google.com/view/ibrahim-almuslimani}{https://sites.google.com/view/ ibrahim-almuslimani}.

\section{Numerical experiments}
\label{sec:num}
We will compare ARKC with {IMPRKC}, PIROCK, and PRKC as they were shown to outperform the other stabilized methods for advection--diffusion--reaction problems where the reaction term is not stiff.

\subsection{Linear 1D advection--diffusion problem}
An excellent example to compare the performance of ARKC with {IMPRKC}, PIROCK, and PRKC is  the following 1-dimensional advection--diffusion equation with periodic boundary conditions 
\begin{equation}\label{eq:adperiodic}
	\begin{split}
		\partial_tu+a\partial_xu&=\partial_x^2u,\\
		u(x,0)&=\sin(2\pi x),\\
		u(0,t)&=u(1,t),
	\end{split}
\end{equation}
where $a$ is a positive constant ($P_e=a$), $x\in[0,1]$, and $t\geq0$. We discretize the space interval $[0,1]$ to a uniform grid $\{x_k\}_{k=0}^{N}$ with $x_k=kh$, and $h=1/N$. We use second order central differences for the advection and the diffusion terms. We denote by $u_k(t)$ the approximation of $u(x_k,t)$, and the periodic boundary conditions propose that $u_0(t)=u_N(t)$. The eigenvalues of the obtained matrix are 
$$
\lambda_k=\frac2{h^2}(\cos(2k\pi h)-1)-\frac{ia}h\sin(2k\pi h), \quad k=1,\dots,N,
$$
and are located in an ellipse in the left half-plane $\IC^-$, which makes the problem typical for the comparison of the three schemes.

\begin{table}[th]
	\centering
	\begin{tabular}{|c|c|c|c|c|c|c|}
		\hline
		Method & $a$   & $Steps$ & $F_D$ evals & $F_A$ evals &  $s\_{max}$ & $L_\infty$ error  at $t=1/2$\\
		\hline
		PRKC&$0.1$ & $14|71$ &  $ 862|1860$    &  $ 56|284$  & $126|106$ & $4.8\times10^{-4}|3\times10^{-7}$ \\
		PIROCK&$0.1$ &   $13|237$    &     $789|3654$       &    $39|711$   & $150|104$    &      $1.8\times10^{-3}|7.5\times10^{-7}$         \\
		{IMPRKC}&$0.1$&$15|83$&$953|2150$&$953|2150$&$139|118$&$3.4\times 10^{-4}|3.3\times10^{-7}$\\
		\textbf{ARKC}&$0.1$ &   $14|79$    &     $886|2098$       &    $42|237$     &    $145|97$    &      $4.3\times10^{-4}|3.3\times 10^{-7}$         \\
		\hline
		PRKC&$0.5$   &   $27|77$    &    $1341|2067$ &  $108|308$ &    $57|57$    &   $8.7\times10^{-9}|3.8\times10^{-9}$  \\
		PIROCK&$0.5$ &   $13|237$    &     $789|3648$       &    $39|711$   & $150|104$    &      $1.8\times10^{-3}|7.3\times10^{-7}$         \\
		{IMPRKC}&$0.5$&$15|83$&$953|2152$&$953|2152$&$112|118$&$3.5\times 10^{-4}|2\times10^{-7}$\\
		\textbf{ARKC}&$0.5$ &   $13|79$    &  $909|2132$  &   $39|237$&$142|117$    &  $2.5\times10^{-4}|2.2\times10^{-7}$      \\
		\hline
		PRKC&$1$&$47|89$&$1803|2342$&$188|356$&$40|40$&$9\times10^{-10}|6.1\times10^{-10}$  \\
		PIROCK& $1$&$13|237$&$849|3648$&$39|711$&$200|104$&$1.7\times10^{-4}|6.6\times10^{-7}$  \\
		{IMPRKC}&$1$&$15|84$&$949|2147$&$949|2147$&$138|121$&$3.7\times 10^{-4}|4.5\times 10^{-7}$\\
		\textbf{ARKC}& $1$&$11|74$&$896|2104$&$33|222$&$194|153$&$2\times10^{-4}|3.6\times10^{-7}$  \\
		\hline
		PRKC&$2$&$90|120$&$2575|2893$&$360|480$&$29|29$& $1.2\times10^{-10}|1\times10^{-10}$  \\
		PIROCK& $2$&$13|237$&$934|3720$&$39|711$&$200|150$&$2\times10^{-6}2.5\times10^{-8}$  \\
		{IMPRKC}&$2$&$15|86$&$942|2174$&$942|2174$&$141|108$&$4.5\times 10^{-4}|2.8\times 10^{-7}$\\
		\textbf{ARKC}&$2$&$10|56$&$995|2267$&$30|168$&$228|172$& $4.8\times10^{-5}|1.8\times10^{-7}$ \\
		\hline
		PRKC&$5$&$222|229$&$3980|4049$&$888|916$&$18|18$&   $4.1\times10^{-11}|4\times10^{-11}$  \\     
		PIROCK& $5$&$14|238$&$1043|3826$&$42|714$&$182|150$&$1.8\times10^{-5}1.7\times10^{-7}$  \\
		{IMPRKC}&$5$&$16|103$&$966|2350$&$966|2350$&$148|111$&$5.5\times 10^{-4}|1.6\times 10^{-7}$\\
		\textbf{ARKC}&$5$&$12|59$&$1272|2764$&$36|177$&$237|184$& $1.9\times10^{-6}|2.9\times10^{-8}$ \\ 
		\hline 
		PRKC&$10$&$442|453$&$5738|5818$&$1768|1812$&$13|13$&   $6.2\times10^{-11}|6\times10^{-11}$  \\  
		PIROCK& $10$&$24|246$&$1574|4086$&$72|738$&$125|125$&$1.3\times10^{-3}|1.2\times10^{-5}$  \\
		{IMPRKC}&$10$&$23|149$&$1227|2819$&$1227|2819$&$106|111$&$9\times 10^{-4}|2.1\times 10^{-6}$\\
		\textbf{ARKC}&$10$&$15|84$&$1359|3207$&$45|252$&$234|160$& $5.4\times10^{-6}|7.3\times10^{-8}$ \\    
		\hline     
		PRKC&$12$&$530|543$&$6355|6436$&$2120|2172$&$12|12$&   $6.4\times10^{-11}|6.2\times10^{-11}$  \\
		PIROCK& $12$&$28|255$&$1750|4306$&$84|765$&$104|104$&$1.1\times10^{-2}|4.3\times10^{-6}$  \\
		{IMPRKC}&$12$&$27|170$&$1353|3036$&$1353|3036$&$92|92$&$2.1\times 10^{--3}|3.1\times 10^{-6}$\\
		\textbf{ARKC}&$12$&$18|104$&$1557|3593$&$54|312$&$196|150$& $3.5\times10^{-5}|4.3\times10^{-7}$ \\      
		\hline     
	\end{tabular}
	\captionof{table}{Comparison of ARKC, PRKC, and PIROCK for the linear advection--diffusion problem \eqref{eq:adperiodic}. The numbers on the left correspond to $tol=10^{-2}$, while those on the right correspond to $tol=10^{-5}$.}
	\label{tableadper}
\end{table}

For the numerical experiments, we take $N=150$, $t\in[0,1/2]$, and we fix $Atol=Rtol=tol$. The initial step is fixed to $10^{-3}$. The number of rejected steps is always very small and thus neglected. We can clearly see in Table \ref{tableadper} that PRKC can compete with our scheme ARKC only in the very small Peclet number regime, while For moderate and large Peclet number, ARKC is notably better. On the other hand, the cost of ARKC is close to that of PIROCK for moderate $P_e$ with small advantage for the latter when using large tolerance $10^{-2}$ . However, ARKC becomes cheaper and more accurate at the same time for large values of $P_e$ (in our experiment for $P_e=10$ and $12$).  For small tolerance $10^{-5}$ ARKC outperforms PIROCK in all Peclet number regimes. This is expected because the damping (and so the vertical with of the stability region) is fixed for PIROCK (when $F_A\neq0$, PIROCK has $\delta_s=0.43s^2$ and $\alpha_s\simeq0.53s+0.5$). The adaptivity of ARKC allows to continuously increase the vertical width of its stability region which lets it to be more flexible with respect to the change in $P_e$. One should not forget that the explicit availability of ARKC coefficients helps a lot in making the scheme adaptive with respect to the change in $P_e$, the feature that is missing in PIROCK.

{In the small Peclet number regime, IMPRKC performs very similar to standard RKC method, which is expected, and thus it is outperformed by ARKC and the other two schemes. For large Peclet numbers, IMPRKC outperforms PRKC and its cost is close to PIROCK, while ARKC performs much better due to the very small number of evaluations of the advection term compared to that of IMPRKC.}

The smaller number of steps needed in ARKC compared to the three other methods, leads to a significantly lower number of $F_A$ evaluations no matter how big is the number of stages.

\subsection{Burgers equation with a nonlinear reaction term}

As an example of a PDE with variable Peclet number, we consider the following Burgers equation with a  nonlinear reaction term and periodic boundary conditions

\begin{equation}\label{eq:burgers}
	\begin{split}
		\partial_tu+10u\partial_xu&=\partial^2_xu+\sin(u^2),\quad (x,t)\in[0,1]\times[0,1/2],\\
		u(x,0)&=1+\sin(2\pi x),\\
		u(0,t)&=u(1,t) \quad \forall \, t\in[0,1/2].
	\end{split}
\end{equation}
We discretize the above equation in space using second order central differences with $\Delta x=10^{-2}$, into $N+1$ grid points with $N=100$. We calculate a reference solution using the Radau IIA method of order 5 \cite{HaW96}.

\begin{figure}[h]
	\centering
	\includegraphics[width=0.5\linewidth]{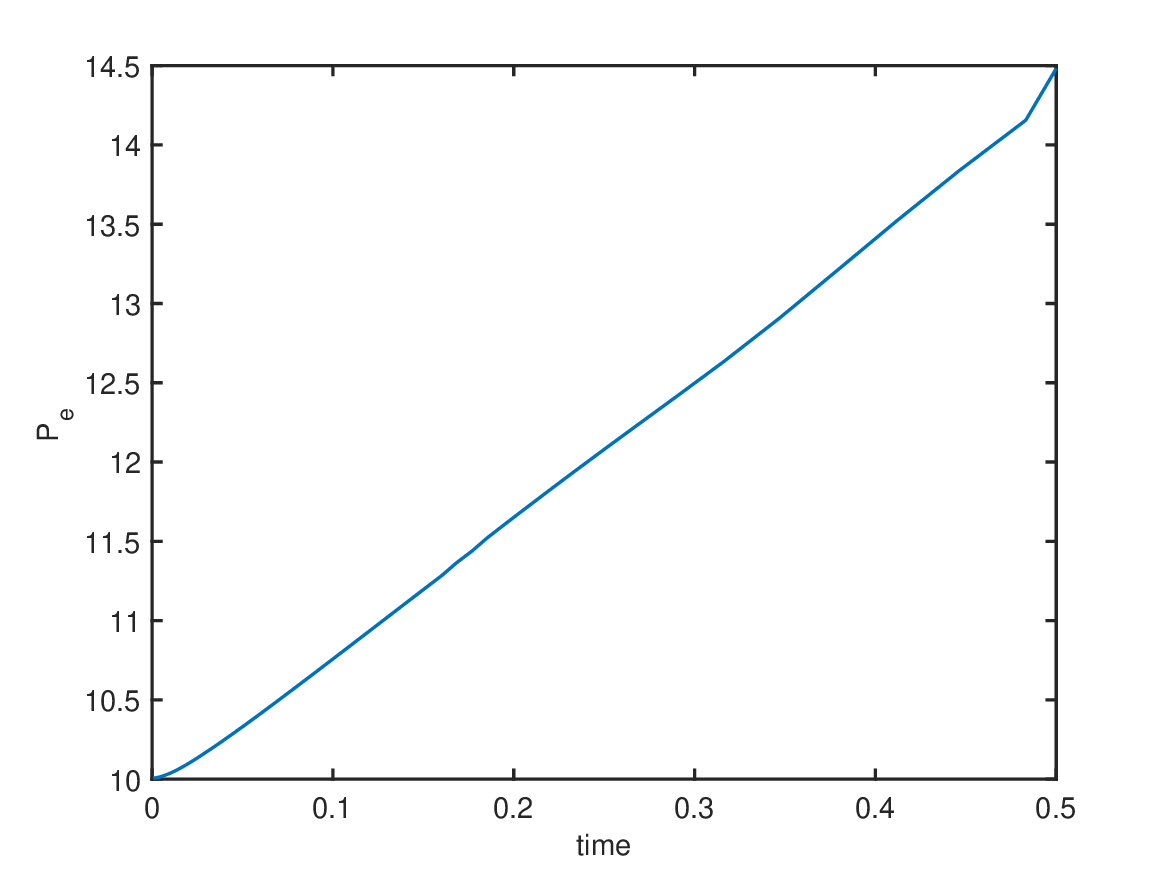}
	\caption{Change in Peclet number with respect to time in problem \eqref{eq:burgers}.}
	\label{fig:pe}
\end{figure}
In Figure \ref{fig:pe}, we plot the Peclet number of equation \eqref{eq:burgers} as a function of time. we see that it is variable and of quite large magnitude.

\begin{figure}[H]
	\centering
	\includegraphics[width=0.75\linewidth]{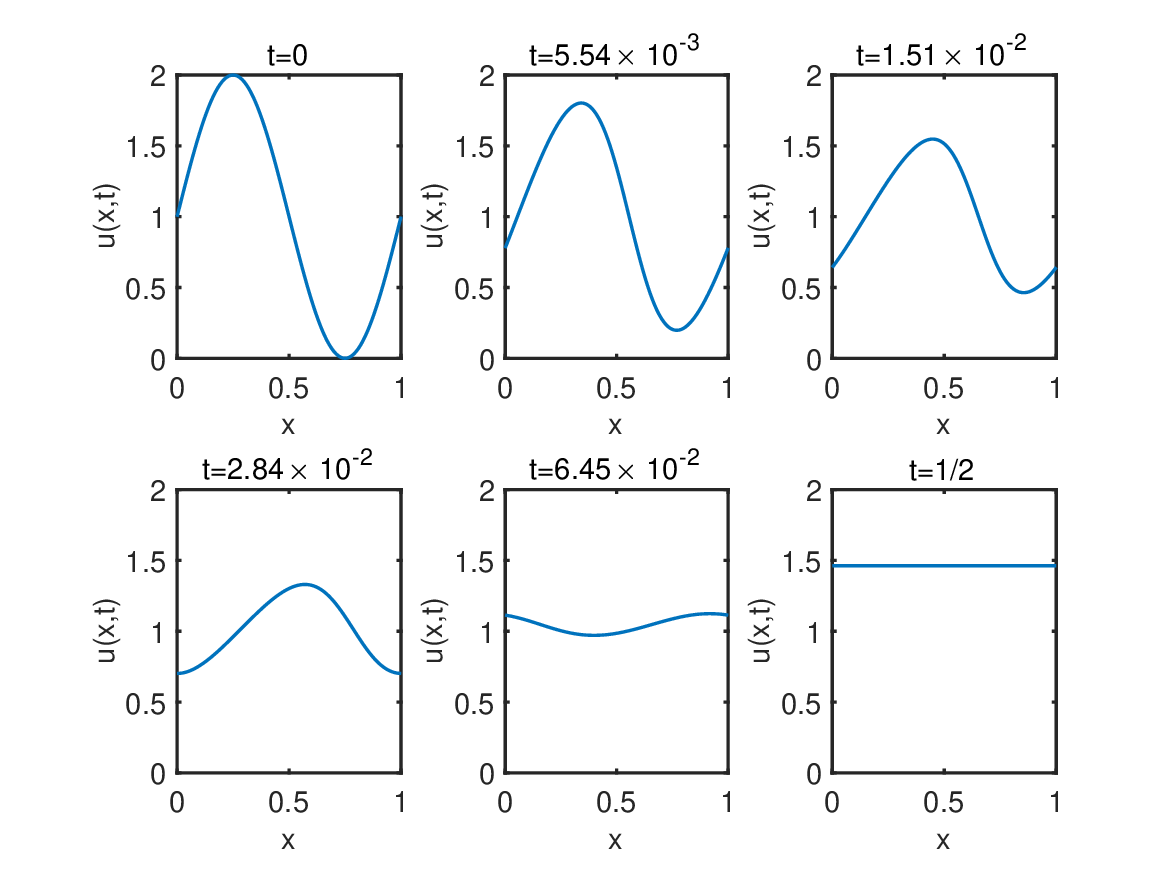}
	\caption{Solution of equation \eqref{eq:burgers} obtained using ARKC method \eqref{eq:ARKC} at different time moments.}
	\label{fig:sol}
\end{figure}
Figure \ref{fig:sol} shows the solution $u(x,t)$ of equation \eqref{eq:burgers} obtained using ARKC method \eqref{eq:ARKC} at different time moments.

\begin{figure}[htp]
	\centering
	\begin{subfigure}[t]{0.49\textwidth}
		\includegraphics[width=1\linewidth]{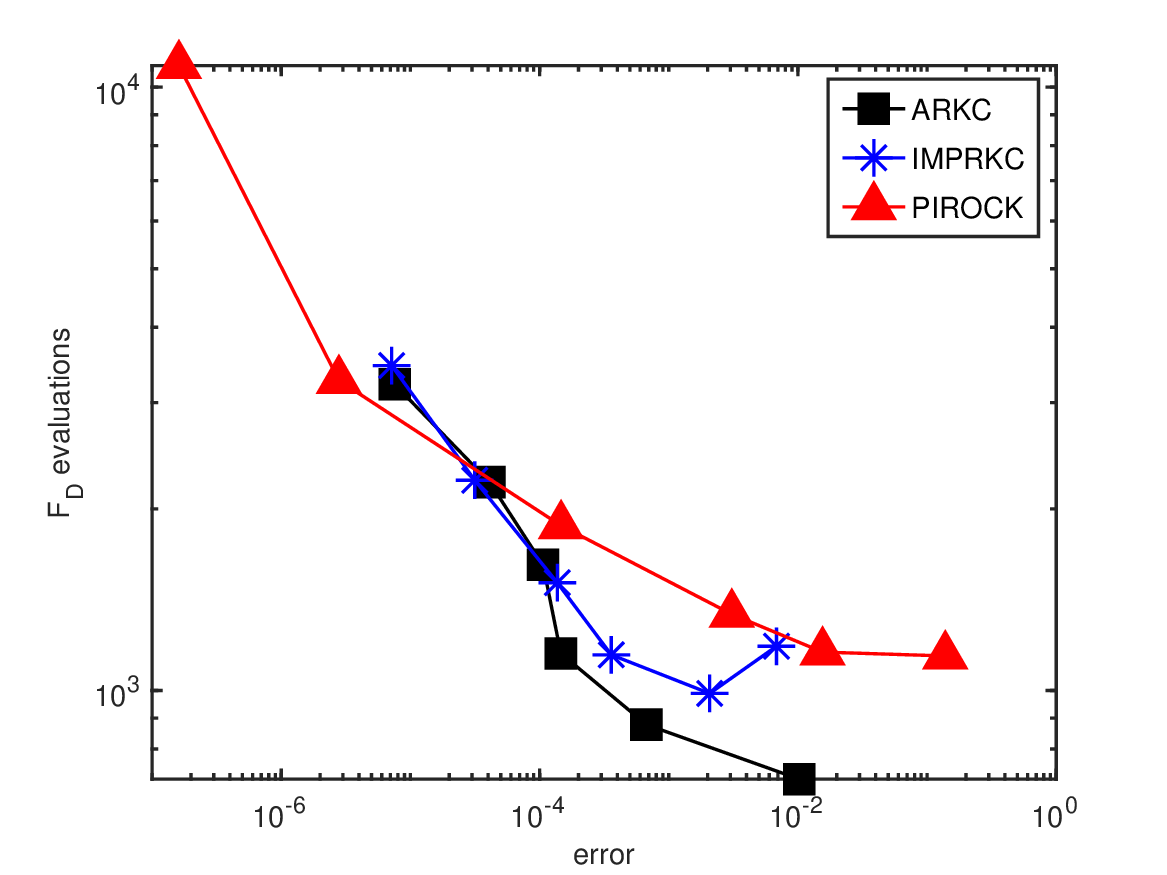}
		\caption{$F_D$ evaluations}
	\end{subfigure}
	\begin{subfigure}[t]{0.49\textwidth}
		\includegraphics[width=1\linewidth]{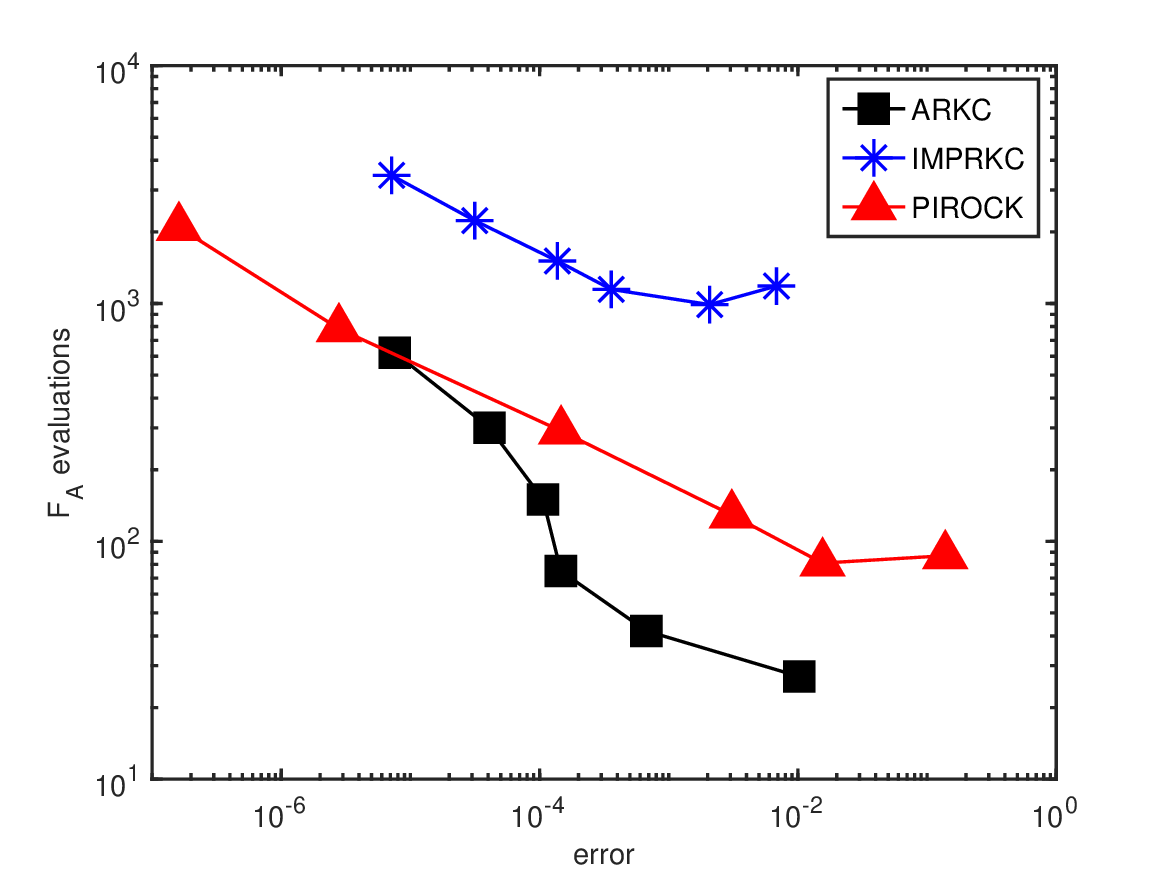}
		\caption{$F_A$ evaluations}
	\end{subfigure}\\
	\caption{Comparison between ARKC, IMPRKC, and PIROCK in terms of cost for problem \eqref{eq:burgers}.}
	\label{fig:costburgers}
\end{figure}

In Figure \ref{fig:costburgers}, we compare the number of functions evaluations needed to obtain a given accuracy for the solution of the Burgers equation \eqref{eq:burgers} using 3 different numerical methods: ARKC, {IMPRKC}, and PIROCK. The results are obtained for $tol=10^{-r},\, r=1,\dots,6$. The advantage of our scheme ARKC \eqref{eq:ARKC} is very clear. {For IMPRKC, the number of $F_D$ evaluations is reasonable, while that of $F_A$ evaluations is very large compared to ARKC and PIROCK}. For PIROCK, the fixed damping increases the number of steps and of functions evaluations with respect to ARKC. The flexibility of ARKC gives it remarkable advantage over the other schemes. We can also see that even when the number of $F_D$ evaluations for ARKC is close to that of the other two schemes, the number of evaluations of  $F_A$ containing the nonlinear advection and reaction terms stays much lower, that is because of the low number of time steps needed.

\section{Conclusion}\label{sec:conclusion}
{In this paper, we have constructed a fully adaptive second order explicit stabilized Runge--Kutta--Chebyshev time integrator for advection--diffusion--reaction PDEs, called ARKC. The new scheme is implemented using an algorithm that is able to adaptively choose the step size, the number of stages and the damping parameter of the method according to the Peclet number. The new scheme is shown to outperform existing methods in the literature for the same type of problems. This high performance is a result of the full adaptivity of Algorithm \ref{algo}, in particular, the adaptive damping that allows significant control of the form of the stability region as a function of the Peclet number.}

\paragraph{Acknowledgements}
The author is grateful to Gilles Vilmart for helpful discussions and comments.


%

\bibliographystyle{abbrv}
\bibliography{abd_biblio,HLW,complete}

\def\cprime{$'$} \def\cprime{$'$} \def\cprime{$'$}
\begin{thebibliography}{10}

\bibitem{Abd01}
A.~Abdulle.
\newblock {\em Chebyshev methods based on orthogonal polynomials}.
\newblock PhD Thesis, University of Geneva, Department of Mathematics.
  University of Geneva, 2001.

\bibitem{Abd02}
A.~Abdulle.
\newblock Fourth order {C}hebyshev methods with recurrence relation.
\newblock {\em SIAM J. Sci. Comput.}, 23(6):2041--2054, 2002.

\bibitem{Abd02b}
A.~Abdulle.
\newblock {ROCK2} and {ROCK4}: software for stiff differential equations
  (discretized parabolic problems).
\newblock {\em Codes available under \texttt{http://anmc.epfl.ch/}}, 2002.

\bibitem{Abd13c}
A.~Abdulle.
\newblock {\em Explicit Stabilized Runge--Kutta Methods}, pages 460--468.
\newblock Encyclopedia of Applied and Computational Mathematics, Springer
  Berlin Heidelberg, 2015.

\bibitem{AAV18}
A.~Abdulle, I.~Almuslimani, and G.~Vilmart.
\newblock Optimal explicit stabilized integrator of weak order 1 for stiff and
  ergodic stochastic differential equations.
\newblock {\em SIAM/ASA J. Uncertain. Quantif.}, 6(2):937--964, 2018.

\bibitem{AbdS20}
A.~Abdulle and G.~R. de~Souza.
\newblock Explicit stabilized multirate method for stiff stochastic
  differential equations.
\newblock {\em arXiv:2010.15193}, 2020.

\bibitem{AbL08}
A.~Abdulle and T.~Li.
\newblock {S-ROCK} methods for stiff {I}to {SDEs}.
\newblock {\em Commun. Math. Sci.}, 6(4):845--868, 2008.

\bibitem{AbM01}
A.~Abdulle and A.~Medovikov.
\newblock Second order chebyshev methods based on orthogonal polynomials.
\newblock {\em Numer. Math.}, 90(1):1--18, 2001.

\bibitem{AbV12a}
A.~Abdulle and G.~Vilmart.
\newblock P{IROCK}: a swiss-knife partitioned implicit-explicit orthogonal
  {R}unge-{K}utta {C}hebyshev integrator for stiff diffusion-advection-reaction
  problems with or without noise.
\newblock {\em J. Comput. Phys.}, 242:869--888, 2013.

\bibitem{AVZ13}
A.~Abdulle, G.~Vilmart, and K.~C. Zygalakis.
\newblock Weak second order explicit stabilized methods for stiff stochastic
  differential equations.
\newblock {\em SIAM J. Sci. Comput.}, 35(4):A1792--A1814, 2013.

\bibitem{A20}
I.~Almuslimani.
\newblock {\em Explicit Stabilized Methods for Stiff Stochastic Differential
  Equations and Stiff Optimal Control Problems}.
\newblock University of Geneva. PhD thesis, 2020.

\bibitem{AV19}
I.~Almuslimani and G.~Vilmart.
\newblock Explicit stabilized integrators for stiff optimal control problems.
\newblock {\em SIAM J. Sci. Comput.}, 43(2):A721--A743, 2021.

\bibitem{B71}
M.~Bakker.
\newblock Analytical aspects of a minimax problem.
\newblock 1971.
\newblock Technical Note TN 62 (in Dutch), Mathematical centre, Amsterdam.

\bibitem{HLW06}
E.~Hairer, C.~Lubich, and G.~Wanner.
\newblock {\em Geometric numerical integration}, volume~31 of {\em Springer
  Series in Computational Mathematics}.
\newblock Springer-Verlag, Berlin, second edition, 2006.
\newblock Structure-preserving algorithms for ordinary differential equations.

\bibitem{HNW93}
E.~Hairer, S.~N{\o{}}rsett, and G.~Wanner.
\newblock {\em Solving Ordinary Differential Equations I. Nonstiff Problems},
  volume~8.
\newblock Springer Verlag Series in Comput. Math., Berlin, 1993.

\bibitem{HaW96}
E.~Hairer and G.~Wanner.
\newblock {\em Solving ordinary differential equations II. Stiff and
  differential-algebraic problems}.
\newblock Springer-Verlag, Berlin and Heidelberg, 1996.

\bibitem{HV03}
W.~Hundsdorfer and J.~Verwer.
\newblock {\em Numerical solution of time-dependent
  advection-diffusion-reaction equations}, volume~33 of {\em Springer Series in
  Computational Mathematics}.
\newblock Springer-Verlag, Berlin, 2003.

\bibitem{TJ07b}
R.~Jeltsch and M.~Torrilhon.
\newblock Flexible stability domains for explicit {R}unge-{K}utta methods.
\newblock In {\em Some topics in industrial and applied mathematics}, volume~8
  of {\em Ser. Contemp. Appl. Math. CAM}, pages 152--180. Higher Ed. Press,
  Beijing, 2007.

\bibitem{KA12}
D.~I. Ketcheson and A.~J. Ahmadia.
\newblock Optimal stability polynomials for numerical integration of initial
  value problems.
\newblock {\em Commun. Appl. Math. Comput. Sci.}, 7(2):247--271, 2012.

\bibitem{SSV98}
B.~Sommeijer, L.~Shampine, and J.~Verwer.
\newblock {RKC}: an explicit solver for parabolic {PDEs}.
\newblock {\em J. Comput. Appl. Math.}, 88:316--326, 1998.

\bibitem{SV80}
B.~P. Sommeijer and J.~G. Verwer.
\newblock {\em A performance evaluation of a class of
  {R}unge-{K}utta-{C}hebyshev methods for solving semidiscrete parabolic
  differential equations}.
\newblock Afdeling Numerieke Wiskunde [Department of Numerical Mathematics],
  91. Mathematisch Centrum, Amsterdam, 1980.

\bibitem{TX20}
X.~Tang and A.~Xiao.
\newblock Improved runge–kutta–chebyshev methods.
\newblock {\em Mathematics and Computers in Simulation}, 174:59--75, 2020.

\bibitem{TJ07a}
M.~Torrilhon and R.~Jeltsch.
\newblock Essentially optimal explicit {R}unge-{K}utta methods with application
  to hyperbolic-parabolic equations.
\newblock {\em Numer. Math.}, 106(2):303--334, 2007.

\bibitem{VHS80}
P.~Van~der Houwen and B.~Sommeijer.
\newblock On the internal stage runge-kutta methods for largem-values.
\newblock {\em Z Angew Math Mech}, 60:479--485, 1980.

\bibitem{HoS80}
P.~J. van~der Houwen and B.~P. Sommeijer.
\newblock On the internal stability of explicit, {$m$}-stage {R}unge-{K}utta
  methods for large {$m$}-values.
\newblock {\em Z. Angew. Math. Mech.}, 60(10):479--485, 1980.

\bibitem{VHS90}
J.~Verwer, W.~Hundsdorfer, and B.~Sommeijer.
\newblock Convergence properties of the runge-kutta-chebyshev method.
\newblock {\em Numer. Math.}, 57:157--178, 1990.

\bibitem{VS04}
J.~G. Verwer and B.~P. Sommeijer.
\newblock An implicit-explicit {R}unge-{K}utta-{C}hebyshev scheme for
  diffusion-reaction equations.
\newblock {\em SIAM J. Sci. Comput.}, 25(5):1824--1835, 2004.

\bibitem{VSH04}
J.~G. Verwer, B.~P. Sommeijer, and W.~Hundsdorfer.
\newblock R{KC} time-stepping for advection-diffusion-reaction problems.
\newblock {\em J. Comput. Phys.}, 201(1):61--79, 2004.

\bibitem{Zb11}
C.~J. Zbinden.
\newblock Partitioned {R}unge-{K}utta-{C}hebyshev methods for
  diffusion-advection-reaction problems.
\newblock {\em SIAM J. Sci. Comput.}, 33(4):1707--1725, 2011.

\end{thebibliography}

\appendix
\section{Damping and number of stages for some choices of Peclet number}
\label{app}

\begin{center}
	\begin{tabular}{ |c|c|c|c|c|c|c|c|c|c|c|c|} 
		\hline
		$s$ & $<11$& $<21$ & $<31$ & $<41$ & $<51$ & $<61$ & $<71$ & $<81$ & $<91$ & $<101$ \\
		\hline
		$\eta$ & $0.15$ & $ 0.6$ &$1$&$1.4$&$1.7$&$2.1$&$2.4$&$2.7$&$3$ &$3.3$\\
		\hline
		$s$ & $<121$& $<141$& $<161$ & $<181$ & $<201$ & $<251$ & $<301$ & $<401$ & $<501$&\\
		\hline
		$\eta$ &$3.7$ & $4.1$ & $4.5$ &$4.9$&$5.3$&$6$&$6.6$&$7.7$&$8.8$&\\
		\hline
	\end{tabular}
	\captionof{table}{$1/2<P_e'\leq1,\,1/4<\rho_A/\sqrt{\rho_D}\leq1/2$.}
			\label{table1}
\end{center}

\begin{center}
	\begin{tabular}{ |c|c|c|c|c|c|c|c|c|c|} 
		\hline
		$s$ & $<11$& $<21$ & $<31$ & $<41$ & $<51$ & $<61$ & $<71$ & $<81$\\
		\hline
		$\eta$  & $0.7$ & $ 1.5$ &$2.3$&$2.9$&$3.5$&$4$&$4.5$&$4.9$\\
		\hline
		$s$ & $<91$ & $<101$ & $<141$ & $<181$ & $<251$ & $<301$ & $<401$ & $<501$\\
		\hline
		$\eta$  &$5.2$ &$5.5$&$6.7$ & $7.7$ & $8.8$ &$9.8$&$11$&$12$\\
		\hline
	\end{tabular}
	\captionof{table}{$1<P_e'\leq3/2,\,1/2<\rho_A/\sqrt{\rho_D}\leq3/4$.}
			\label{table1p5}
\end{center}

\begin{center}
	\begin{tabular}{ |c|c|c|c|c|c|c|c|c|c|c|c|} 
		\hline
		$s$ & $<11$& $<21$ & $<31$ & $51$& $<71$ & $<111$ & $<151$&$<311$&$501$\\
		\hline
		$\eta$ & $1$ & $ 2.5$ &$3.5$&$4.8$&$6$&$7.8$&$9$&$12.5$&$15$\\
		\hline
	\end{tabular}
	\captionof{table}{$3/2<P_e'\leq2,\,3/4<\rho_A/\sqrt{\rho_D}\leq1$.}
			\label{table2}
\end{center}

\begin{center}
	\begin{tabular}{ |c|c|c|c|c|c|c|c|c|c|c|c|} 
		\hline
		$s$ & $<11$& $<21$ & $<31$ & $51$& $<71$ & $<111$ & $<151$&$<311$&$501$\\
		\hline
		$\eta$ & $2$ & $ 3.8$ &$5$&$6.8$&$8$&$10.4$&$12$&$16$&$19$\\
		\hline
	\end{tabular}
	\captionof{table}{$2<P_e'\leq2\sqrt{2},\,1<\rho_A/\sqrt{\rho_D}\leq\sqrt{2}$.}
			\label{table2s2}
\end{center}

\begin{center}
	\begin{tabular}{ |c|c|c|c|c|c|c|c|c|c|c|c|} 
		\hline
		$s$ & $<11$& $<31$ & $<71$ & $<151$ & $<311$ & $<501$\\
		\hline
		$\eta$ & $4$ & $ 9$ &$13.5$&$18$&$23$&$27$\\
		\hline
	\end{tabular}
	\captionof{table}{$P_e'>2\sqrt{2},\,\rho_A/\sqrt{\rho_D}>\sqrt2$.}
			\label{tablelarge}
\end{center}

{One can get more tables and increase the adaptivity of the algorithm with respect to damping. This will for sure increase the performance of the method. However, the method performs already very well with the tables we provided.}

\end{document}